\documentclass[11pt]{amsart}
\usepackage{geometry}                
\usepackage{graphicx}
\usepackage{amssymb}
\usepackage{epstopdf}
\usepackage{amsmath,amsthm}
\usepackage{amscd}
\usepackage{amsfonts, amssymb} 
\usepackage{latexsym}

\newtheorem{theorem}{Theorem}
\newtheorem{corollary}[theorem]{Corollary}
\newtheorem{lemma}{Lemma}

\theoremstyle{definition}

\newcommand \ad{\text{ad}\thinspace}

\newcommand\la{\langle}
\newcommand\ra{\rangle}
\renewcommand{\frak}{\mathfrak}

\title{A generalization of Lazard's elimination theorem}
\author{Elizabeth Jurisich  and Robert Wilson }

\address{Department of Mathematics,  College of Charleston,
Charleston SC 29424, jurisiche@cofc.edu}

\address{Department of Mathematics Rutgers
University, 110 Frelinghuysen Road, Piscataway, NJ 08854, rwilson@math.rutgers.edu }

\begin{document}

\maketitle

\begin{abstract}Using the classical Lazard's elimination theorem,
we obtain a decomposition theorem for Lie algebras defined by
generators and relations of a certain type. This is a preprint version of the paper appearing in Communications in Algebra Volume 32, Issue 10, 2004. 
\end{abstract}

\section{Introduction}

This paper grew out of, and has its main application in, the
theory of generalized Kac-Moody Lie algebras over a field $\Phi$. Borcherds 
initiated the study of these algebras in \cite{2}, and applied the
theory in his proof of the Conway-Norton conjectures \cite{3}. 
Generalized Kac-Moody Lie algebras may be defined by generators and
relations (see \cite{7}).  
While studying these algebras it is natural to consider an algebra of the
form $$L(V \oplus W)/I$$
where $L(X)$ denotes the free Lie algebra on the vector space $X$, $V$
and $W$ are vector spaces with $V \cap W = (0)$, $L(V \oplus W)$ is
graded by giving nonzero elements of $V$ degree zero and nonzero
elements of $W$ degree
one, and where $I$ is an ideal generated by homogeneous elements of
degree zero or one. Theorem 1 gives the structure of
such an algebra.

We write $U(\frak r)$ for the universal enveloping algebra of the Lie 
algebra $\frak r$. The algebra $U(\frak r)$ acts on $\frak r$ via the
adjoint action; $ a \cdot b$ denotes the image of $b$ in $\frak r$
under the action of $a \in U(\frak r)$.

Our main result, Theorem 1, generalizes and is proved using the
following theorem of 
Lazard \cite{10}, \cite[Proposition 10]{4}. 
Let $M = U(L(V))\cdot W \subset L(V \oplus W)$.
Note that $M$ is a $U(L(V))$-module and so is an $L(V)$-module.

\begin{theorem}[Lazard's Elimination Theorem]
The ideal of $L(V \oplus W) $ generated by $W$ is isomorphic to $L(M)$
and therefore $L(V \oplus W) \cong L(V) \ltimes L(M)$. \end{theorem}

R. Block has pointed out to us that this theorem is proven in
\cite{4} only for finite dimensional $V$ and $W$.
Work of Block and Leroux \cite{1} shows that the theorem holds in general. 

Several generalizations of Lazard's elimination theorem
 are known: \cite{1} (giving a
general category theoretic result which in the special case of a free
Lie algebra gives Lazard's theorem),  \cite{6} (treating certain 
generalized Kac-Moody Lie algebras), \cite{7}
(treating all generalized Kac-Moody Lie algebras), \cite{11} (treating the
case in which $I$ is generated by elements of degree zero).  All of
these results on Lie algebras are contained in Theorem 1. In the
special case pertinent to the Conway-Norton conjectures this theorem
yields the decomposition appearing in \cite{6}. This decomposition
simplifies part of the proof of the Conway-Norton 
conjectures \cite{6, 9}. \cite{5}, which treats the case in
which $I$ is generated by a collection, $\theta$, of elements of the 
form $[r,s]$ where $r,s \in$ a basis for $V \oplus W$, follows from Theorem 1 
only for certain $\theta$ (those in which, for 
every pair $r,s$, we have $r,s \in V \bigcup W$ and at least one of $r,s$ is in $V$).

Section 2 contains the statement and proof of Theorem 1.  Section 3
discusses applications to generalized Kac-Moody Lie algebras.

\section{Main Result}

If $\frak r$ is a Lie algebra and $ S \subset \frak r$ we let 
$\la  S \ra_\frak r$ denote
the ideal of $\frak r$ generated by $S$. Of course 
$\la  S \ra_\frak r = U(\frak r) \cdot S$.

\begin{lemma} 
If $U \supseteq W$ then $L(U)/{\la W \ra}_{L(U)} \cong L(U /W)$.
\end{lemma}

\begin{proof}
Write $U= V \oplus W$, so $V \cong U /W$. Then 
\begin{align*}
L(U)/ \la W \ra_{L(U)} & =  L(V \oplus W)/\la W \ra_{L(U)} = (L(V) \ltimes \la W
\ra_{L(U)})/\la W \ra_{L(U)} \\
&\cong L(V) \cong L(U/W).
\end{align*}
\end{proof}

Now, as in Lazard's Elimination Theorem, let $M = U(L(V))\cdot W \subseteq L(V \oplus W).$  
Then, as a corollary of that theorem we obtain:

\begin{lemma} \hfill
\begin{enumerate}
\item $U(L(V \oplus W)) = U(L(M))U(L(V))$
\item $U(L(V\oplus W)) = U(L(V)) + U(L(M)) M U(L(V))$.
\end{enumerate}
\end{lemma}

\begin{proof}
By the elimination theorem
$L(V \oplus W) = L(V) \ltimes L(M)$ 
(where we identify $\la W \ra_{L(V \oplus W)}$ 
with $L(M)$), so (1) follows by the
Poincar\'{e}-Birkhoff-Witt Theorem. Also, as $\Phi $ is the base field
\begin{align*}
U(L(M))&= \Phi + U(L(M))L(M)\\
&=\Phi + U(L(M))M
\end{align*} 
so (2) follows {}from (1).
\end{proof}

Now let $A \subset L(V) \subset L(V\oplus W)$, $B \subset M \subset L(M)
\subset L(V \oplus W)$. Thus, if $L(V \oplus W)$ is
graded by giving nonzero elements of $V$ degree zero and nonzero
elements of $W$ degree
one, then $A$ is an arbitrary subspace of elements of degree 
zero and $B$ is an arbitrary subspace of elements of degree one.
Write $\frak g = L(V)/ \la A \ra_{L(V)}$, 
$M_1 = \{ [M, \la A\ra_{L(V)}] + U(L(V))\cdot B \}$ and $N =
M/M_1$.

The following theorem, our main result, gives the structure of the quotient algebra 
$L(V \oplus W)/ \la A, B\ra_{L(V \oplus W)}$.

\begin{theorem} \hfill
\begin{enumerate}
\item The space $N$ is a $\frak g$-module.
\item $L(N)$ is isomorphic as a 
$\frak g$-module to the ideal of $L(V \oplus W)/ \la
A, B\ra_{L(V \oplus W)}$ generated by the image of $W$.
\item $L(V \oplus W)/ \la A, B\ra_{L(V \oplus W)} \cong \frak g
\ltimes L(N)$.
\end{enumerate}
\end{theorem}

\begin{proof}
To prove (1), note $M= U(L(V))\cdot W$ is an $L(V)$-module as is
$U(L(V))\cdot B$. Since $M$ and $\la A\ra_{L(V)}$ are $L(V)$-modules,
so is $[M, \la A \ra_{L(V)}]$. Thus $N$ is an $L(V)$-module. Since
$[\la A\ra_{L(V)}, M] \subset M_1$, $N$ is a $\frak g$-module.

We now prove (2):
\begin{align*}
\la A, B\ra_{L(V \oplus W)} &= \la A \ra_{L(V \oplus W)} + \la B
\ra_{L(V\oplus W)} \\
&= U(L(V \oplus W))\cdot A + U(L(V \oplus W))\cdot B.
\end{align*}

By Lemma 2 this is equal to
\begin{align*}
U(L(V))\cdot A &+ U(L(M))MU(L(V))\cdot A + U(L(M))U(L(V))\cdot B \\
&= \la A \ra_{L(V)} + \la [ M, \la A \ra_{L(V)}] \ra_{L(M)}
  + \la U(L(V))\cdot B\ra_{L(M)}\\
&= \la A \ra_{L(V)} + \la M_1 \ra_{L(M)}.
\label{E:*}\tag{*}\end{align*} 
Now the ideal of $L(V \oplus W)/ \la A, B \ra_{L(V \oplus W)}$
generated by $W$ is 
\begin{align*}
(L(M) 
&+  \la A, B \ra_{L(V \oplus W)})/ \la A, B \ra_{L(V \oplus W)}\\
&\cong L(M)/(L(M) \cap  \la A, B \ra_{L(V \oplus W)}).
\end{align*}
By equation (\ref{E:*}) this is equal to
 $L(M)/ \la M_1 \ra_{L(M)}$. By
Lemma 1 this is $L(M/M_1) = L(N)$.

Furthermore, $L(V \oplus W) / \la A, B \ra_{L(V \oplus W)} \cong (L(V)
\ltimes L(M))/\la A, B \ra_{L(V \oplus W)}$. By equation (*) this is
isomorphic to 
$$L(V)/ \la A \ra_{L(V)} \ltimes  L(M)/ \la M_1 \ra_{L(M)}
= \frak g \ltimes L(N),$$
so we have proven (3).
\end{proof}

\section{Applications}

 Let $I$ be an index set  which is finite or countably infinite and
let $R \subset I \times I$. 
 Let $\frak n$ be the Lie algebra with generators $X= \{x_i | i\in I\}$
and relations
$(\ad x_i)^{n_{ij} }x_j$ for $(i,j)\in R$. We may assume that $R$ does
not contain  
diagonal elements $(i,i)\in I$ because $[x_i,x_i]=0 $ in $L(X)$.

If $J\subset I$, let $\frak n_J$
denote the subalgebra of $\frak n$ generated by the $x_i$ for $i \in
J$. 
Theorem 1 gives:

\begin{theorem} Let $\frak n, I, R$ be as above. 
Suppose that for some choice of $S,T \subset I$, $I = S\cup T$
(disjoint union),  and $i,j \in T$ with $i \neq j$
implies $(i,j) \notin R$. 
Then $\frak n \cong \frak n_S \ltimes L(U(\frak n_S)\cdot W)$
where $W$ denotes the vector space spanned by the $x_i$ for $i \in T$.
\end{theorem}

\begin{proof}
Let $V$ be the vector space with basis $x_i , i \in S$. Take 
$A = \{ (\ad x_i)^{n_{ij}}~x_j | i, j \in S\}$ and 
$B = \{ (\ad x_i)^{n_{ij}}x_j | i \in S, j \in T\}$. Then Theorem 1 
gives the above decomposition, where we write $N$ as the $\frak n_S$-module in
$L(V \oplus W)/\langle A \rangle_{L(V \oplus W)}$ generated by $W$.
\end{proof}

Let $\frak l$ be a generalized Kac-Moody algebra associated to a symmetrizable 
matrix $(a_{ij})_{i,j\in I}$.
By Proposition 1.5 \cite{6} one has
$\frak l = \frak n^+ \oplus \frak h \oplus \frak n^-$.
Because the radical (the maximal graded ideal not intersecting 
$\frak h$) is zero (see \cite{6},\cite{7}), the subalgebras  
$\frak n^\pm$ can be
written as $\frak n$ above, choosing $X= \{x_i = e_i | i \in I\}$ 
for $\frak n^+$ and $X= \{x_i = f_i | i \in I\}$ for $\frak n^-$ where 
$e_i$ and $f_i$, $i \in I$ are the Chevalley generators of $\frak
l$. The Serre relations 
$(\ad x_i)^{n_{ij}}x_j=0$ occur whenever $a_{ii}>0$, or when 
$a_{ii}\leq 0$ and $a_{ij} =0$.
If we take $R$ to be the set corresponding to the occurrence of Serre
relations and $S, T$ as in Theorem 2, then applying Theorem 2 to both
$\frak n^+$ and $\frak n^-$ gives Theorem 3.19 of \cite{7}:

\begin{corollary}{Corollary 3}
Let $\frak l$ be a generalized Kac-Moody algebra associated to a symmetrizable 
matrix $(a_{ij})_{i,j\in I}$. Let $R$ denote the
set $\{(i,j)| a_{ii}>0 \text{ or } a_{ii}\leq 0 \text{ and } a_{ij} =0\}\subset
I \times I$. Choose $S,T$ so that
$I = S \cup T$ (disjoint union), and
$i,j \in T$ implies $(i,j) \notin R$.
Let $\frak l_{1}$ be the subalgebra of $\frak l$ 
generated by the $e_i$ and $f_i$ with $i \in S$. 
Then 
$\frak l = \frak u^+ \oplus (\frak l_{1}+ \frak h) \oplus \frak u^-$, 
where
$\frak u^-$ is the free Lie algebra 
on the direct sum of the standard highest weight
$\frak l_{1}$-modules
${\mathcal U}(\frak n^-_{S})\cdot f_j$ for $j \in T$ and 
$\frak u^+$ is the free Lie algebra on the direct sum of the
standard lowest weight $\frak l_{1}$-modules 
${\mathcal U}(\frak n^+_{S})\cdot e_j$ for $j \in T$.
\end{corollary}

Theorem 5.1 of \cite{6} is a special case of this Corollary. As noted in 
\cite{7} one can iterate this decomposition until $\frak l_1$ is a 
semi-simple or Kac-Moody subalgebra. The results of 
\cite{5} on free partially commutative Lie algebras may be obtained as the case 
where all $a_{ii}<0$. 

This decomposition is used in \cite{6} to obtain the denominator
identity for the algebra $\frak l$ {}from the identity for the
subalgebra $\frak l_1$. Conversely, one can prove Corollary 3 using the
denominator and character formulas for generalized Kac-Moody algebras.
(This is the proof in \cite{7}.) Other applications include computing
the homology of the Lie algebra over a standard module, and
determining a class of completely reducible modules \cite{8}.


\end{document}